\documentclass[preprint,11pt]{elsarticle}
           \usepackage[T2A]{fontenc}   
           \usepackage[utf8]{inputenc}
                     \usepackage[english]{babel}

\makeatletter
\def\ps@pprintTitle{%
  \let\@oddhead\@empty
  \let\@evenhead\@empty
  \let\@oddfoot\@empty
  \let\@evenfoot\@oddfoot
}
\makeatother

\usepackage{amsfonts, amsmath, amscd}
\usepackage[psamsfonts]{amssymb}
\usepackage{amssymb}

\usepackage{amsmath}
\usepackage{amsfonts}
\usepackage{amssymb}
\usepackage{amsthm}
\usepackage[pdftex,unicode]{hyperref}
\usepackage{indentfirst} 
\usepackage{amscd}
\usepackage[all,cmtip]{xy}
\usepackage{mathtools}
\usepackage{comment}
\usepackage{tikz-cd}
\usetikzlibrary{babel}

\newtheorem{proposition}{Proposition}
\newtheorem{theorem}{Theorem}
\newtheorem*{theorem*}{Theorem}
\newtheorem{lemma}{Lemma}
\newtheorem*{lemma*}{Lemma}
\newtheorem{cor}{Corollary}

\theoremstyle{definition}

\newtheorem{problem}{Problem}

\theoremstyle{remark}
\newtheorem{note}{Remark}
\newtheorem*{note*}{Remark}

  \def\R{\mathbb R}

\def\N{\mathbb N}

\def\om{\omega}

\def\supp{\operatorname{supp}}

\def\be{\beta}

\def\ph{\varphi}
\def\de{\delta}

\def\cl#1{\overline{#1}}
\def\clx#1#2{\overline{#2}^{\,#1}}

\def\bt{{\beta}}
\def\bt{{\beta}\mskip 1}
\def\bt{{\beta}\mskip 0.4mu}

\def\es{\varnothing}

\def\sset#1{\{#1\}}

\def\set#1{\bbset#1\eeset}
\def\bbset#1:#2\eeset{\{#1\,:\,#2\}}

\def\bbsett#1:#2\eesett{\{#1\,:\,\text{#2}\}}

\def\ibbset#1:#2\ieeset{(#1)_{#2}}

\def\cB{{\mathcal B}}

\def\cB{{\mathcal B}}

\def\cU{{\mathcal U}}

\def\B{\mathbb D}

\def\cV{{\mathcal V}}

\newcommand\restrA[2]{{
  \left.\kern-\nulldelimiterspace 
  #1 
  \vphantom{\big|} 
  \right|_{#2} 
  }}

\def\oo#1/{$O_{#1}$}

\def\ep{\varepsilon}

\def\gd/{$G_\delta$}

\def\si{\sigma}

\def\pini/{{п.и.н.и}}
\def\vum/{{в.у.м}}
\def\cum/{{ч.у.м}}
\def\lum/{{л.у.м}}

\def\term#1{{\it #1}}

\def\cccc/{$ccc$}

\def\:{\colon}

\def\see\cite#1{(see~\cite{#1})}
\def\seee\cite[#1]#2{(see~\cite[#1]{#2})}

\def\B{\mathcal{B}}
\def\Ba{\mathcal{B}a}

\def\bt{\beta}
\def\om{\omega}
\def\si{\sigma}
\def\cB{\mathcal{B}}

\begin{document}

\begin{frontmatter}

\title{A New Approach to Universal Measurability\tnoteref{corr}}
\tnotetext[corr]{This is a corrected version of the paper: statements of the previous version that turned out to be incorrect have been removed. Details of the corrections are given in the last section, ``About this version''.}

\author{Reznichenko E.A.}
\address{Department of Mathematics, Lomonosov Moscow State University, Moscow, Russia}
\author{Sadovnichiy Yu.V.}
\address{Department of Mathematics, Lomonosov Moscow State University, Moscow, Russia}

\begin{abstract}
V. Fedorchuk, A. Chigogidze, and T. Banakh in 2003 and V. Bogachev in 2024 posed the following questions: (i) is it true that $P_\tau(X)$ is $C$-embedded in $P_\si(X)$; (ii) is it true that $P_R(X)$ is $C$-embedded in $P_R(\beta X)$ if and only if $X$ is pseudocompact, where $P_\si$, $P_\tau$, and $P_R$ are the functors of probability measures on the space $X$ that are, respectively, $\si$-additive on the Baire $\si$-algebra, $\tau$-additive, and Radon? The answers to these questions are negative. In this paper the questions are put into a general framework: for spaces of measures $Q_p(X)\subset Q_q(X)$ of various types we study when $Q_p(X)$ is $C$-embedded ($C^\ast$-embedded) in $Q_q(X)$. It is proved that for spaces of alternating measures $M_p(X)\subset M_q(X)$ the $C$-embeddedness, the $C^\ast$-embeddedness and the $G_\delta$-density of $M_p(X)$ in $M_q(X)$ are equivalent, and a criterion in terms of sequences of continuous functions is given.

The question of $C$-embedding of measure spaces is an extension of the question of coincidence of measure spaces, which is a development of the classical concepts of universally measurable and universal 
measure zero sets.
\end{abstract}

\begin{keyword} 
measure spaces
\sep
probability measures
\sep 
Radon measures
\sep 
universally measurable sets
\sep 
universal measure zero sets
\end{keyword}

\end{frontmatter}
\section{Introduction}\label{sec:intro}

This paper considers various ways of extending the concept of universally measurable and universal measure zero sets from separable metrizable spaces to the class of Tychonoff spaces.

Let $X\subset \R$.\footnote {In a slightly more general setting, we consider subsets of Polish spaces (=complete separable metric spaces).} A set $X$ is called {\em universally measurable} if $X$ is $\mu$-measurable for any finite measure $\mu$ on the Borel sets of the line.
A set $X$ is called {\em universal measure-zero} if $X$ has measure zero with respect to
any continuous probability measure on Borel sets.

The classical extension of these concepts from the class of separable metrizable spaces to the class of Hausdorff spaces is Radon spaces and universally negligible spaces. These classes of spaces are inconvenient for applications, since not every compactum is a Radon space.

Given the Riesz representation theorem, for applications in analysis, functional analysis, and topology, it is more convenient to consider measures on the Baire $\si$-algebra of Tychonoff spaces. In \cite{BanakhChigogidzeFedorchuk2003}, the concept of universally measurable spaces was extended from separable metrizable spaces to the class of Tychonoff spaces. In this case, all Hausdorff compact spaces turned out to be universally measurable.

Further, we assume that all spaces are Tychonoff.

\subsection{Measures on a Tychonoff Space}

For a space $X$, we denote by $\B(X)$ and $\Ba(X)$ the $\si$-algebras of Borel and Baire sets of $X$.

For a compact space $K$, let $M(K)$ denote the dual of $C(K)$ with the weak${}^\ast$ topology, where $C(K)$ is the Banach space of continuous functions on $K$ with the $\sup$-norm.
By the Riesz representation theorem, the functionals $\mu$ in $M(K)$ are identified with regular Borel measures on $K$.
Denote by $U(K)$ the unit ball in $M(K)$, by $M^+(K)$ the nonnegative functionals in $M(K)$, and by $P(K)=\{\mu\in M^+(K): \mu(K)=1\}$ the probability measures on $K$.

For a Tychonoff space $X$, the most important measure spaces on $X$ are: $M_a(X)$ --- finitely additive measures on $\Ba(X)$; $M_\si(X)$ --- countably additive measures on $\Ba(X)$; $M_\tau(X)$ --- $\tau$-additive measures on $\B(X)$; $M_R(X)=M(X)$ --- Radon measures on $\B(X)$; $M_\bt(X)$ --- compactly supported measures on $\B(X)$; $M_d(X)$ --- discrete measures on $\B(X)$.
These measure spaces are identified with subsets of $M(\bt X)$: for $k\in\sset{d,\be,R,\tau,\si,a}$ and $\mu\in M(\be X)$, the inclusion $\mu\in M_k(X)$ holds if the condition $(M_k)$ is satisfied:
\begin{itemize}
\item[$(M_d)$]
$\mu$ is concentrated on some countable subset of $X$;
\item[$(M_\be)$]
$\mu$ is concentrated on some compact subset of $X$;
\item[$(M_R)$]
$\mu$ is concentrated on some $\si$-compact subset of $X$;
\item[$(M_\tau)$]
$|\mu|(K)=0$ for all compact $K\subset \bt X\setminus X$;
\item[$(M_\si)$]
$|\mu|(K)=0$ for all compact $K\subset \bt X\setminus X$ of type $G_\de$ in $\bt X$;
\item[$(M_a)$]
$\mu$ is arbitrary.
\end{itemize}
Note that
\[
M_d(X)\subset M_R(X),\ \ \ \ \ \ M_\be(X)\subset M_R(X) \subset M_\tau(X) \subset M_\si(X) \subset M_a(X)=M(\bt X).
\]

\subsection{Generalization of universally measurable and universal measure zero Sets}
\label{ssec:gum}

For a set $A\subset \R$, the classical notions of {\em universally measurable} and {\em universal measure zero} sets $A$ are topological and are characterized as $M_\si(A)=M_R(A)$ and $M_\si(A)=M_d(A)$.

For $p,q\in\{a,\si,\tau,R,\bt,d\}$, the space $X$ is called
\begin{itemize}
\item
{\em universally $[p,q]$-measurable} if $M_p(X)=M_q(X)$;
\item
{\em universally $p$-measurable} if $M_p(X)=M_R(X)$;
\item
{\em universally $p$-measure zero} if $M_p(X)=M_d(X)$.
\end{itemize}

Section 2 of the paper \cite{BanakhChigogidzeFedorchuk2003} is devoted to universally $[\si,\tau]$-measurable and universally $\si$-measurable spaces (which are called universally measurable in \cite{BanakhChigogidzeFedorchuk2003}).

\begin{note}\label{n:1}
Let us formulate known results on this topic.

\smallskip
(i)
$K$-Analytic (in particular, Souslin) spaces are universally $\si$-measurable \cite[7.14.125]{Bogachev2007Measure_theory}.

\smallskip
(ii)
A space $X$ is universally $\tau$-measurable if and only if $X$ is $\mu$-measurable in $\bt X$ for any measure $\mu\in P(\bt X)$ \cite[Example 7.14.22]{Bogachev2007Measure_theory}.

\smallskip
(iii)
A paracompact space $X$ is universally $[\si,\tau]$-measurable if and only if $X$ does not contain an Ulam-measurable discrete closed subset
(see~\cite[Theorem I.28]{Varadarajan1965measures} and \cite[Section 2]{BanakhChigogidzeFedorchuk2003}). In particular, a Lindelöf space is a universally $[\si,\tau]$-measurable space \cite[Corollary 2.2]{BanakhChigogidzeFedorchuk2003}.

\smallskip
(iv)
For a space $X$, the following conditions are equivalent \cite[Theorem 1.22]{Reznichenko2025betap}:
(a) $X$ is pseudocompact;

(b) $X$ is a universally $[\si,a]$-measurable space.
\end{note}

The following assertion is a strengthening of Remark \ref{n:1}(ii).

\begin{theorem}\label{t:gum:1}
For a space $X$, the following statements are equivalent:
\begin{itemize}
\item[{\rm (i)}]
$X$ is universally $\tau$-measurable;
\item[{\rm (ii)}]
$X$ is $\mu$-measurable in $\bt X$ for any measure $\mu\in P(\bt X)$;
\item[{\rm (iii)}]
there exists an embedding of space $X$ into some space $Y$ such that
$X$ is $\mu$-measurable in $Y$ for any measure $\mu\in P(\bt Y)$;
\item[{\rm (iv)}]
for any embedding of space $X$ into any space $Y$,
$X$ is $\mu$-measurable in $Y$ for any measure $\mu\in P(\bt Y)$.
\end{itemize}
\end{theorem}

The Rudin--Knowles theorem \cite{Rudin1957,Knowles1967} states that for a compact space $X$,
any Radon measure on $X$ is discrete (i.e., $M_R(X)=M_d(X)$, which is equivalent to $X$ being of universally $R$-measure zero) if and only if $X$ is a scattered compactum. The Rudin--Knowles theorem implies the following assertion.

\begin{proposition}\label{p:umnR}
A space $X$ is universally $R$-measure zero if and only if every compact subset of $X$ is scattered.
\end{proposition}

\begin{proposition}\label{p:umnR-ast}
Let $p\in\{a,\sigma,\tau\}$.
A space $X$ is universally of $p$-measure zero if and only if $X$ is universally $p$-measurable and
every compact subset of $X$ is scattered.
\end{proposition}

\subsection{Functional Generalizations of universally measurable and Universally Measure Zero Sets}\label{ssec:func}

Recall that a subset $Y$ of $X$ is {\em $C$-embedded} ({\em $C^\ast$-embedded}) in $X$ if every (bounded) continuous function on $Y$ extends to a continuous function on $X$.

We consider the subfunctors $M^+_p(X)=M_p(X)\cap M^+(\bt X)$, $U_p(X)=M_p(X)\cap U(\bt X)$, $U^+_p(X)=U_p(X)\cap M^+(\bt X)$, $P_p(X)=M_p(X)\cap P(\bt X)$ of non-negative measures, measures in norm not exceeding one, probability measures, respectively, of the functor $M_p(X)$.
For
$p,q\in\{a,\si,\tau,R,\bt,d\}$ and
$Q\in \{M,M^+,U,U^+,P\}$, $X$ is called
\begin{itemize}
\item
{\em universally $Q_{[p,q]}$-measurable ($Q^*_{[p,q]}$-measurable)} if $Q_p(X)$ is $C$-embedded ($C^\ast$-embedded) in $Q_q(X)$;
\item
{\em universally $Q_{p}$-measurable ($Q^*_{p}$-measurable)} if $X$ is universally $Q_{[R,p]}$-measurable ($Q^\ast_{[R,p]}$-measurable);
\item
{\em universally $Q_{p}$-measure ($Q^*_{p}$-measure) zero} if $X$ is universally $Q_{[d,p]}$-measurable ($Q^\ast_{[d,p]}$-measurable).
\end{itemize}

In the definitions and theorems on $C$-embeddings it is always assumed that $Q_p(X)\subset Q_q(X)$, i.e.\ $p\preccurlyeq q$ in the order $d,\bt\ \preccurlyeq\ R\ \preccurlyeq\ \tau\ \preccurlyeq\ \si\ \preccurlyeq\ a$ (the indices $d$ and $\bt$ are not comparable).

Using the introduced concepts, some problems in measure theory and topology (see \cite[Question 1.21, Question 1.29]{Reznichenko2025betap}, \cite[Section 5, Problem 5.5]{BanakhChigogidzeFedorchuk2003}, \cite{Bogachev2024-1ru,Bogachev2024dan-ru}) are special cases of the following general question.

\smallskip
\begin{problem}
For $p,q\in\{a,\si,\tau,R,\bt,d\}$ and $Q\in \{M,M^+,U,U^+,P\}$, describe universally
${[p,q]}$-measurable, universally $Q_{[p,q]}$-measurable, and universally $Q^\ast_{[p,q]}$-measurable spaces.
\end{problem}

In \cite[Section 5]{BanakhChigogidzeFedorchuk2003}, the problem is posed: to describe universally $P_{[\tau,\si]}$-measurable and universally $P_\si$-measurable spaces.
And the question arises \cite[Problem 5.5]{BanakhChigogidzeFedorchuk2003}: is it true that every space $X$ is universally $P_{[\tau,\si]}$-measurable?
In \cite{Reznichenko2025betap}, a negative answer is given to this question.

In \cite{Bogachev2024-1ru,Bogachev2024dan-ru}, universally $P_{a}^\ast$-measurable spaces are studied, and it is proved that if $X$ is universally $P_{a}^\ast$-measurable space, then $X$ and $P(X)$ are pseudocompact spaces.
In \cite[Theorem 1.9]{Reznichenko2025betap}, a converse of Bogachev's theorem was obtained: a space $X$ is universally $P_{a}^\ast$-measurable if and only if $P(X)$ is a pseudocompact space.
The question was posed: is it true that if $X$ is pseudocompact, then $X$ is universally $P_{a}^\ast$-measurable? In \cite{Reznichenko2025betap}, a negative answer is given to this question (see Examples 1.13 and 1.15 of \cite{Reznichenko2025betap}).

\begin{note}
Let us formulate known results on this topic.

\smallskip
(i)
If $X$ is a universally $\si$-measurable space and the finite powers $X^n$ of $X$ have a countable Souslin number (in particular, $X$ is a Souslin space), then $X^\kappa$ is a universally $P_\si$-measurable space for any cardinal $\kappa$ \cite[Corollary 5.2]{BanakhChigogidzeFedorchuk2003}.

\smallskip
(ii)
Any $AE(0)$-space is a universally $P_\si$-measurable space \cite[Corollary 5.3]{BanakhChigogidzeFedorchuk2003}.

\smallskip
(iii)
For a space $X$, the following conditions are equivalent \cite[Theorem 1.22]{Reznichenko2025betap}:
(a) $X$ is pseudocompact;
(b) $X$ is a universally $[\si,a]$-measurable space;
(c) $X$ is a universally $P_{[\si,a]}^\ast$-measurable space.

\smallskip
(iv)
For a space $X$, the following conditions are equivalent \cite[Theorem 1.27]{Reznichenko2025betap}:
(a) $X$ is pseudo $\om$-bounded;
(b) $X$ is a universally $P_{[\be,a]}^\ast$-measurable space.

\smallskip
(v)
For a space $X$, the following conditions are equivalent \cite[Theorem 1.9]{Reznichenko2025betap}:
(a) $P(X)$ is pseudocompact;
(b) $X$ is a universally $P_{a}^\ast$-measurable space.

\smallskip
(vi)
For a space $X$ with precaliber $\om_1$, the following conditions are equivalent \cite[Theorem 1.28]{Reznichenko2025betap}:
(a) $P_\tau(X)$ is pseudocompact;
(b) $X$ is a universally $P_{[\tau,a]}^\ast$-measurable space.
\end{note}

\subsection{Alternating Measures}

Probability measure functors have the best categorical properties. Other measure functors are also important \cite{Koumoullis1981, Sadovnichii1999, SadovnichiiFedorchuk1999, Sadovnichii2000, Sadovnichii2007, SadovnichiiFedorchuk2008, Sadovnichii2020}
and are more convenient for some applications.

A subset $M$ of a space $X$ is called {\em $G_\de$-dense} if $M$ intersects every nonempty $G_\de$-set (=the intersection of countably many open sets) in $X$.

A continuous mapping of compact spaces $f: X\to Y$ extends to a continuous mapping
\[
M(f): M(X)\to M(Y),\ M(f)(\mu)(h) = \mu(h \circ f)\text{ for }h\in C(Y).
\]
A continuous mapping of spaces $f: X\to Y$ extends to a continuous mapping $\bt f: \bt X\to \bt Y$ of Stone–Cech extensions.
Then
\[
M(\bt f)(M_p(X))\subset M_p(Y),
\]
where $p\in\{a,\si,\tau,R,\bt,d\}$.

The following theorem is the main result of the paper.

\begin{theorem}\label{t:main:1}
Let $X$ be a space and $p,q\in\{a,\si,\tau,R,\bt,d\}$ with $M_p(X)\subset M_q(X)$. The following conditions are equivalent:
\begin{enumerate}
\item[{\rm (i)}]
$X$ is a universally $M_{[p,q]}$-measurable space;
\item[{\rm (ii)}]
$X$ is a universally $M^\ast_{[p,q]}$-measurable space;
\item[{\rm (iii)}]
$M_p(X)$ is $G_\de$-dense in $M_{q}(X)$;
\item[{\rm (iv)}]
For every measure $\mu \in M_q(X)$ and every sequence of functions $(h_n)_n\subset C(\bt X)$ there exists $\nu\in M_p(X)$ such that $\mu(h_n)=\nu(h_n)$ for all $n$.
\end{enumerate}
\end{theorem}

In the previous version the following two statements were given as corollaries of this theorem. They are incorrect; see the last section, ``About this version''.

\begin{cor}\label{c:main:1}
\emph{(Incorrect statement.)} If $X$ is pseudocompact, then $X$ is universally $M_a^\ast$-measurable.
\end{cor}

\begin{cor}\label{c:main:2}
\emph{(Incorrect statement.)} Any space is universally $M_{[\tau,\si]}$-measurable.
\end{cor}

\subsection{Probabilistic, Non-Negative, and Unit-Bounded Measures}
\label{ssec:pmpu}

There exists a pseudocompact locally compact space $X$ that is neither universally $P_\si^\ast$-measurable nor $P_{[\tau,\si]}$-measurable \cite[Example 1.15]{Reznichenko2025betap}, which answers negatively the questions in \cite{BanakhChigogidzeFedorchuk2003,Bogachev2024-1ru}.

The following assertion follows directly from the definitions.

\begin{proposition}\label{p:pmpu:1}
Let $X$ be a space, $p,q\in\{a,\si,\tau,R,\bt,d\}$, and $Q\in \{M,M^+,U,U^+,P\}$.
If $X$ is a universally $Q_{[p,q]}$-measurable space, then $X$ is a universally $Q_{[p,q]}^\ast$-measurable space.
\end{proposition}

\begin{theorem}\label{t:pmpu:2}
Let $X$ be a space, $p,q\in\{a,\si,\tau,R,\bt,d\}$, and $Q\in \{M^+,U,P\}$.
If $X$ is a universally $Q_{[p,q]}^\ast$-measurable space, then $X$ is a universally $M_{[p,q]}$-measurable space.
\end{theorem}

\begin{theorem}\label{t:pmpu:3}
Let $X$ be a space and $p,q\in\{a,\si,\tau,R,\bt,d\}$.
\begin{enumerate}
\item[{\rm (i)}]
A space $X$ is a universally $P_{[p,q]}$-measurable space
if and only if $X$ is a universally $M^+_{[p,q]}$-measurable space.
\item[{\rm (ii)}]
A space $X$ is a universally $P^\ast_{[p,q]}$-measurable space
if and only if $X$ is a universally $M^{+\ast}_{[p,q]}$-measurable space.
\end{enumerate}
\end{theorem}

\begin{problem}
Let $p,q\in\{a,\si,\tau,R,\bt,d\}$.
How are the following classes of spaces related?
\begin{itemize}
\item
universally $P_{[p,q]}$-measurable spaces;
\item
universally $P^\ast_{[p,q]}$-measurable spaces;
\item
universally $U_{[p,q]}$-measurable spaces;
\item
universally $U^\ast_{[p,q]}$-measurable spaces;
\item
universally $M_{[p,q]}$-measurable spaces.
\end{itemize}
\end{problem}


\section{Preliminary Information, Definitions, and Notation}\label{sec:defs}

Let $\om=\sset 0 \cup \N$ be the first infinite ordinal, and $\om_1$ be the first uncountable ordinal. By space, we mean a Tychonoff space.

\subsubsection{Definitions from Measure Theory}
Let $\mu$ be a finite (i.e., $\mu(X)<\infty$) $\si$-additive measure on the Borel $\si$-algebra $\B(X)$ of $X$.
Let $\mu=\mu^+-\mu^-$ be the Hahn--Jordan decomposition. The measure $|\mu|=\mu^++\mu^-$ is called the {total variation} of $\mu$.
A measure $\mu$ is called \term{Radon} if for every Borel set $B\subset X$ and for every $\ep>0$, there exists a compact $K\subset B$ such that $|\mu|(B\setminus K)<\ep$.
\term{The support} $\supp\mu$ of a measure $\mu$ is the smallest closed set $F\subset X$ for which $|\mu|(X\setminus F)=0$. A measure $\mu$ is called \term{discrete} if $|\mu|(X\setminus C)=0$ for some countable $C\subset X$. A measure $\mu$ is called \term{continuous} if $\mu(\sset x)=0$ for every $x\in X$.
A measure $\mu$ is called \term{$\tau$-additive} if
\[
|\mu|(\bigcup\cU) = \sup \set{ |\mu|(\bigcup\cV) : \cV\subset \cU,\ |\cV|<\om }
\]
for any family $\cU$ of open subsets of $X$.

For a non-negative measure $\mu$, we denote
\begin{align*}
\mu^*(A) &= \inf\set{\mu(B): A\subset B\in \cB(X)},
&
\mu_*(A) &= \sup\set{\mu(B): A\supset B\in \cB(X)}
\end{align*}
\term{external} and \term{internal} measures of a set $A\subset X$.

A countably additive measure $\mu$ on the $\si$-algebra $\cB(X)$ of Borel sets in $X$ is called a \term{Borel} measure.
We denote the set of probability Borel measures on $X$ by $P_b(X)$.
Let $\mu\in P_b(X)$.
A probability measure $\mu$ is called a \term{Radon} measure if, for every Borel set $B\subset X$ and for every $\ep>0$, there exists a compact $K\subset B$ such that $\mu(B\setminus K)<\ep$.
We denote the set of probability $\tau$-additive measures on $X$ by $P_\tau(X)$. Note that $P(X)\subset P_\tau(X)$.

Let $Y$ be a space and $X\subset Y$. The set $P_b(X)$ is naturally embedded in $P_b(Y)$, and a measure $\mu\in P_b(X)$ corresponds to a measure $\tilde\mu\in P_b(Y)$ such that $\tilde\mu(B)=\mu(B\cap X)$ for $B\in\cB(Y)$. We assume that $P_b(X)\subset P_b(Y)$. If $Y$ is compact, then $P(X)\subset P_\tau(X)\subset P(Y)$.

\begin{proposition}[\cite{Banakh1995ru}]\label{p:defs:2}
Let $Y$ be a compact space and $X\subset Y$.
Then $w(X)=w(P(X))=w(P_\tau(X))$ and
\begin{align*}
P(X) &= \set{\mu \in P(Y): \mu_*(X)=1},
&
P_\tau(X) &= \set{\mu \in P(Y): \mu^*(X)=1}.
\end{align*}
The embedding of $P(X)$ and $P_\tau(X)$ into $P(Y)$ is a topological embedding.
\end{proposition}

Let $X$ be a space.
Then
\begin{align*}
M_k(X) &= \set{t\mu- q\nu: \mu,\nu\in P_k(X),\ 0\leq t,q},
\\
M^+_k(X) &= \set{t\mu: \mu\in P_k(X),\ 0\leq t},
\\
U_k(X) &= \set{t\mu- q\nu: \mu,\nu\in P_k(X),\ t,q\geq 0,\ t+q\leq 1}.
\end{align*}
for $k\in\sset{d,\be,R,\tau,\si,a}$.
The space $U_k(X)$ is the unit ball in $M_k(X)$.
The spaces $U_k(X)$ and $M_k(X)$ for $k\in\sset{\be,R,\tau}$ were studied in the works \cite{Koumoullis1981,Sadovnichii1999,SadovnichiiFedorchuk1999,Sadovnichii2000,Sadovnichii2007,SadovnichiiFedorchuk2008,Sadovnichii2020}.

\subsubsection{Definitions from Topology}
If $X$ is a space and $M\subset X$, then we denote by $\clx XM$ the closure of $M$ in $X$. If the context makes it clear which $X$ is meant, we write $\cl M$ instead of $\clx XM$.

A space $X$ is said to have a countable Souslin number if every disjoint family of open sets in $X$ is at most countable.

A space $X$ is called \term{$\om$-bounded} if $\cl M$ is compact for every countable $M\subset X$.
A space $X$ is called \term{totally countably compact} if for every infinite $L\subset X$, there exists an infinite countable $M\subset L$ such that $\cl M$ is compact \see\cite{Vaughan1984handbook}.

A space $X$ is called \term{pseudo-$\om$-bounded} if for every sequence $(U_n)_n$ of nonempty open sets, there exists a compact subset $K\subset X$ such that $U_n\cap K\neq\es$ for all $n$ \see\cite{AngoaOrtiz-CastilloTamariz-Mascarua2013,AngoaOrtiz-CastilloTamariz-Mascarua2014,Garcia-FerreiraOrtiz-Castillo2018}.

A space $X$ is called \term{almost pseudo $\om$-bounded} (almost pseu\-do-$\om$-bo\-un\-ded) if, for any sequence $(U_n)_n$ of nonempty open sets, there exists a subsequence $(n_k)_k$ and a compact subset $K\subset X$ such that $U_{n_k}\cap K\neq\es$ for all $k$ \cite[Definition 3.8]{AngoaOrtiz-CastilloTamariz-Mascarua2014}, \cite[Definition 1.4.5]{AACCICTM2018}.

A subset $G\subset X$ of $X$ is called a \term{$G_\de$-set} if it is the intersection of countably many open sets.

A subset $S\subset X$ of $X$ is called \term{$G_\de$-dense} in $X$ if $S$ intersects every nonempty $G_\de$-set. A space $X$ is pseudocompact if and only if $X$ is $G_\de$-dense in its Stone--Cech extension $\bt X$.

A space $X$ is called \term{realcompact} if $X$ embeds closedly in the product of lines $\R^A$. A \term{Hewitt realcompactification} $\nu X$ of $X$ is a realcompact extension of $X$ such that $X$ is dense in $\nu X$ and $X$ is $C$-embedded in $\nu X$ \seee\cite[section 3.11]{EngelkingBookRu}.
The Hewitt extension $\nu X$ is naturally embedded in the Stone–Cech extension $\bt X$: $p\in \nu X\subset \bt X$ if and only if every $G_\de$-set $G\subset \bt X$ containing $p$ intersects $X$. The space $X$ is pseudocompact if and only if $\nu X=\bt X$. The space $X$ is realcompact if and only if $\nu X=X$ \seee\cite[section 3.11]{EngelkingBookRu}.

A subset $F\subset X$ is called a {\em zero set} if $F=f^{-1}(0)$ for some continuous function on $X$.

A space $X$ is called {\em perfectly $\kappa$-normal} if the closure of every open set is a zero set.


\section{Locally Convex Spaces in the Weak Topology}\label{sec:wlcs}

\begin{lemma}\label{l:lvs:rg}
Let $L_w$ be a locally convex linear space $L$ with the weak topology.
The space embeds, naturally and by an isomorphic homeomorphism, into $\R^\Gamma$ as a dense linear subspace, where $\Gamma$ is a Hamel basis of the dual space $L'$.
\end{lemma}
\begin{proof}
For $v\in L_w$ and $g\in \Gamma$, put $\ph(v)(g)=g(v)$. The map $\ph\colon L_w\to \R^\Gamma$ is the desired embedding.
\end{proof}

\begin{proposition}\label{p:lvs:kappa}
Let $L_w$ be a locally convex linear space $L$ with the weak topology.
Then $L_w$ is perfectly $\kappa$-normal.
\end{proposition}
\begin{proof}
By Lemma \ref{l:lvs:rg}, we may assume that $L_w$ is a dense linear subspace of $\R^\Gamma$, where $\Gamma$ is a Hamel basis of the dual space $L'$.
The space $\R^\Gamma$ is perfectly $\kappa$-normal \cite{Shchepin1976}. Since $L_w$ is dense in $\R^\Gamma$, $L_w$ is perfectly $\kappa$-normal.
\end{proof}

\begin{theorem}\label{t:lcs:c-emb}
Let $L_w$ be a locally convex linear space $L$ with the weak topology.
Then, for a dense subset $M\subset L_w$, the following conditions are equivalent:
\begin{enumerate}
\item
$M$ is $C^*$-embedded in $L_w$;
\item
$M$ is $C$-embedded in $L_w$;
\item
$M$ is $G_\delta$-dense in $L_w$.
\end{enumerate}
\end{theorem}
\begin{proof}
By Proposition \ref{p:lvs:kappa}, $L_w$ is perfectly $\kappa$-normal.
The theorem now follows from Lemma 9.9.35, Theorem 6.1.5, Proposition 6.1.6, Theorem 6.1.7, Lemma 8.4.5, and Theorem 9.9.36 of \cite{at2009}.
\end{proof}

\begin{theorem}\label{t:lcs:}
Let $L_w$ be a locally convex linear space $L$ with the weak topology and let $M\subset L_w$ be a dense subset.
For every continuous function $f\colon M\to \R$ there exists a sequence $(g_n)_{n\in\om}\subset L'$ of continuous functionals such that $f=h\circ g$, where
\[
g\colon M\to \R^\om,\ g(v)(n)=g_n(v)
\]
and $h\colon g(M)\to \R$ is some continuous function.
\end{theorem}
\begin{proof}
By Lemma \ref{l:lvs:rg}, we may assume that $L_w$ is a dense linear subspace of $\R^\Gamma$, where $\Gamma$ is a Hamel basis of the dual space $L'$.
Then $M$ is dense in $\R^\Gamma$. By the factorization lemma 0.2.3 of \cite{Arhangelskii1989CpBook}, there exists an at most countable set $G\subset \Gamma$ such that $f=p \circ \pi_G$, where
\[
\pi_G\colon \R^\Gamma\to \R^G,\ \pi_G(s)(v)=s(v)
\]
is the projection and $p\colon \pi_G(M)\to \R$ is some continuous function. Enumerate the set $G$ as $G=\{g_n:n\in\om\}$. The sequence $(g_n)_{n\in\om}$ is the one required.
\end{proof}

\begin{cor}\label{c:lcs:}
Let $L_w$ be a locally convex linear space $L$ with the weak topology.
For every continuous function $f\colon L_w\to \R$ there exists a sequence $(g_n)_{n\in\om}\subset L'$ of continuous functionals such that $f=h\circ g$, where
\[
g\colon L_w\to \R^\om,\ g(v)(n)=g_n(v)
\]
and $h\colon g(L_w)\to \R$ is some continuous function.
\end{cor}


\section{Proofs}\label{sec:proofs}
Let us describe the main steps of the proof.

\begin{proof}[Proof of Theorem \ref{t:main:1}]
We may assume that $p\neq q$.

(i)$\leftrightarrow$(ii)$\leftrightarrow$(iii). The space $M_q(X)$ is a linear subspace of $M(\bt X)=C(\bt X)^\ast$; with the topology induced by the weak$^\ast$ topology it is a locally convex space in the weak topology, and $M_p(X)\subset M_q(X)$. The set $M_d(X)\subset M_p(X)$ is dense in $M(\bt X)$ (finitely supported measures on a dense subset of $\bt X$ are weak$^\ast$ dense), so $M_p(X)$ is dense in $M_q(X)$. Theorem \ref{t:lcs:c-emb} applied to $L_w=M_q(X)$ and $M=M_p(X)$ gives the equivalence.

(iii)$\to$(iv). Let $\mu\in M_q(X)$ and $(h_n)_n\subset C(\bt X)$. The set $\{\nu\in M_q(X):\nu(h_n)=\mu(h_n)\text{ for all }n\}$ is a nonempty $G_\de$-subset of $M_q(X)$, hence it meets $M_p(X)$.

(iv)$\to$(iii). Assume the contrary: $M_p(X)$ is not $G_\de$-dense in $M_{q}(X)$. Then there is a nonempty zero set $G\subset M_{q}(X)\setminus M_p(X)$ (a nonempty $G_\de$-set in a Tychonoff space contains a nonempty zero set); let $G=f^{-1}(0)$ and $\mu\in G$. The dual of $M_q(X)$ with the weak$^\ast$ topology is $C(\bt X)$. By Corollary \ref{c:lcs:} there exists a sequence $(h_n)_n\subset C(\bt X)$ such that $F=\{\nu\in M_q(X):\nu(h_n)=\mu(h_n)\text{ for all }n\}\subset f^{-1}(f(\mu))=G$.
Since $F\subset M_{q}(X)\setminus M_p(X)$, condition (iv) is violated.
\end{proof}

Theorem \ref{t:pmpu:3} follows from the fact that $P_p(X)$ is $C$-embedded ($C^\ast$-embedded) in $P_q(X)$ if and only if $M^+_p(X)$ is $C$-embedded ($C^\ast$-embedded) in $M^+_q(X)$.

\section*{About this version}\label{sec:version}

This is a corrected version of the paper. Its text corresponds to the previous version, which coincides with the paper published in \emph{Vestnik Instituta Matematiki NAN KR}, 2025, no.~2, pp.~21--32; statements that turned out to be incorrect have been removed, and no new results have been added. The further development of the subject, with the corrected statements, a counterexample and complete proofs, is given in the authors' paper \emph{$C$-embeddings of spaces of signed measures} (preprint).

The changes are as follows.
(a)~Condition (v) of Theorem~\ref{t:main:1} is removed: it was stated as the inclusion $M(\bt f)(M_q(X))\subset M_p(Y)$, whereas the proof gives only the equality $M(\bt f)(M_p(X))=M(\bt f)(M_q(X))$, and the inclusion does not imply (i)--(iv).
(b)~Corollaries~1 and~2 are marked as incorrect, and the claims based on them are removed from the abstract and from Section~\ref{ssec:pmpu}. They are false: there is a pseudocompact locally compact space $X$ such that $M_\tau(X)$ is not $C^\ast$-embedded in $M_\si(X)=M(\bt X)$ (see the paper mentioned above).
(c)~The proof of (iii)$\leftrightarrow$(iv) in Theorem~\ref{t:main:1} is completed, and the theorem is stated for $M_p(X)\subset M_q(X)$ (a convention in Section~\ref{ssec:func}).
(d)~The definitions of a discrete measure and of $M^+_k(X)$, $U_k(X)$, $U^+_p(X)$ and the indices in the question of \cite{BanakhChigogidzeFedorchuk2003} ($P_{[\tau,\si]}$ instead of $P_{[\si,\tau]}$) are corrected; misprints are fixed.
(e)~Complete proofs of the results of Section~\ref{ssec:pmpu} are given in the paper mentioned above.

\bibliographystyle{plain}


\begin{thebibliography}{10}

\bibitem{Banakh1995ru}
T. Banakh.
\newblock {The topology of spaces of probability measures, I: Functors $P_\tau$ and $\hat P$.}
\newblock {\em Mat. Studii}, 5(1-2):65–87, 1995.

\bibitem{Bogachev2024dan-ru}
V.~I. Bogachev.
\newblock Compactification of Spaces of Measures and Pseudocompactness.
\newblock {\em Doklady Akademii Nauk}, 518:75–79, 2024.

\bibitem{Bogachev2024-1ru}
V.I. Bogachev.
\newblock On the Compactification of Spaces of Measures.
\newblock {\em Functional Analysis and Its Applications}, 58(1):4--21, 2024.

\bibitem{Reznichenko2025betap}
E.A.~Reznichenko.
\newblock Stone-Cech Extensions of Spaces of Probability Measures, preprint, 2024.

\bibitem{Sadovnichii1999}
Yu.~V. Sadovnichiy.
\newblock {On Some Categorical Properties
of the Functor $U_\tau$}.
\newblock {\em Bulletin of Moscow University.
Series 1: Mathematics. Mechanics}, (3):38--42, 1999.

\bibitem{Sadovnichii2007}
Yu.~V. Sadovnichiy.
\newblock {On topological and categorical properties of the functors $M_\tau$ and $M_R$}.
\newblock {\em Doklady Akademii Nauk}, 414(1):24--27, 2007.

\bibitem{SadovnichiiFedorchuk2008}
Yu.~V. Sadovnichiy and V.~V. Fedorchuk.
\newblock {The functor $M_\tau$, Lipschitz and uniformly
continuous mappings}.
\newblock {\em Bulletin of Moscow University.Series 1: Mathematics. Mechanics}, (1):8--12, 2008.

\bibitem{SadovnichiiFedorchuk1999}
V.~V. Fedorchuk and Yu.~V. Sadovnichiy.
\newblock On Some Topological and Category Properties of Alternating Measures.
\newblock {\em Fundamental and Applied Mathematics}, 5(2):597–618, 1999.

\bibitem{EngelkingBookRu}
R. Engelking.
\newblock {\em General Topology}.
\newblock PWN - Panstwowe Wydawnictwo Naukowe, Warszawa, 1977.

\bibitem{Arhangelskii1989CpBook}
A.~V. Arhangel'skii.
\newblock {\em Topological Function Spaces}.
\newblock Moscow State University, Moscow, 1989; English translation: Kluwer Academic Publishers, Dordrecht, 1992.


\bibitem{AngoaOrtiz-CastilloTamariz-Mascarua2014}
J.~Angoa, Y.F.~Ortiz-Castillo, and A.~Tamariz-Mascar{\'u}a.
\newblock Ultrafilters and properties related to compactness.
\newblock In {\em Topology Proc}, volume~43, pages 183--200, 2014.

\bibitem{AACCICTM2018}
J.~Angoa-Amador, A.~Contreras-Carreto, M.~Ibarra-Contreras, and
  A.~Tamariz-Mascar{\'u}a.
\newblock Basic and classic results on pseudocompact spaces.
\newblock In {\em {Pseudocompact Topological Spaces: A Survey of Classic and
  New Results with Open Problems}}, pages 1--38. Springer, 2018.

\bibitem{at2009}
A. Arhangel'skii and M. Tkachenko.
\newblock {\em {Topological Groups and Related Structures}}.
\newblock Atlantis Press, 2008.

\bibitem{BanakhChigogidzeFedorchuk2003}
T. Banakh, A. Chigogidze, and V. Fedorchuk.
\newblock {On spaces of $\sigma$-additive probability measures}.
\newblock {\em Topology and its Applications}, 133(2):139--155, 2003.

\bibitem{Bogachev2007Measure_theory}
V.~I. Bogachev.
\newblock {\em Measure theory. {Vol}. {I} and {II}}.
\newblock Berlin: Springer, 2007.

\bibitem{BogachevSmolyanovSobolev2017en}
V.~I. Bogachev, O.~G. Smolyanov, and V.I.~Sobolev.
\newblock {\em Topological vector spaces and their applications}.
\newblock Springer, 2017.

\bibitem{Garcia-FerreiraOrtiz-Castillo2018}
S.~Garc{\'\i}a-Ferreira and Y.F.~Ortiz-Castillo.
\newblock Pseudocompactness and ultrafilters.
\newblock In {\em {Pseudocompact Topological Spaces: A Survey of Classic and
  New Results with Open Problems}}, pages 77--105. Springer, 2018.

\bibitem{AngoaOrtiz-CastilloTamariz-Mascarua2013}
Á. Tamariz-Mascarúa J.~Angoa, Y. F. Ortiz-Castillo.
\newblock Compact-like properties in hyperspaces.
\newblock {\em Matematički Vesnik}, 65(253):306--318, 2013.

\bibitem{Knowles1967}
J.D.~Knowles.
\newblock On the existence of non-atomic measures.
\newblock {\em Mathematika}, 14(1):62--67, 1967.

\bibitem{Koumoullis1981}
G. Koumoullis.
\newblock Some topological properties of spaces of measures.
\newblock {\em Pacific Journal of Mathematics}, 96(2):419--433, 1981.

\bibitem{Rudin1957}
Walter Rudin.
\newblock Continuous functions on compact spaces without perfect subsets.
\newblock {\em Proceedings of the American Mathematical Society}, 8(1):39--42,
  1957.

\bibitem{Sadovnichii2000}
Yu~V. Sadovnichiy.
\newblock {On some categorical properties of the functor $U_R$}.
\newblock {\em Topology and its Applications}, 107:131--145, 2000.

\bibitem{Sadovnichii2020}
Y.~V. Sadovnichiy.
\newblock On alternating measure functors.
\newblock {\em Topology and its Applications}, 2020.

\bibitem{Shchepin1976}
E.~V. Shchepin.
\newblock Topology of limit spaces of uncountable inverse spectra.
\newblock {\em Russian mathematical surveys}, 31(5):155, 1976.

\bibitem{Varadarajan1965measures}
V.S.~Varadarajan.
\newblock Measures on topological spaces.
\newblock {\em Amer. Math. Soc. Transl}, 48(2):161--228, 1965.

\bibitem{Vaughan1984handbook}
J.~E. Vaughan.
\newblock Countably compact and sequentially compact spaces.
\newblock In {\em Handbook of set-theoretic topology}, pages 569--602.
  Elsevier, 1984.

\end{thebibliography}

\end{document}